\documentclass[10pt]{article}
\usepackage[all]{xy}
\usepackage{amsmath}
\usepackage{amsfonts}
\usepackage{amssymb}
\usepackage{amscd}
\usepackage{amsthm}
\usepackage{latexsym}
\usepackage{amsbsy}
\usepackage{graphicx}
\newtheorem{lemma}{Lemma}[section]

\newtheorem{theorem}{Theorem}[section]

\newtheorem{definition}{Definition}[section]

\newcommand{\C}{\mathbb{C}}

\newcommand{\delb}{\overline{\partial}}
\newcommand{\del}{\partial}
\newcommand{\At}{A_{\theta}}
\newcommand{\Ats}{A_{\theta}^{\infty}}
\newcommand{\SR}{\mathcal S(\mathbb R)}
\begin{document}

\title{A Riemann-Roch theorem for the noncommutative two torus}
\author {  Masoud Khalkhali$^1$,  Ali Moatadelro$^2$\\
\emph{Department of Mathematics, University of Western Ontario, Canada}}
\date{}
\maketitle{}

\begin{abstract} We prove the analogue of the Riemann-Roch formula  for  the noncommutative two torus $ A_{\theta} = C(\mathbb{T}_{\theta}^2)$ equipped
with an arbitrary translation invariant complex structure and a Weyl factor represented by a positive element $k\in C^{\infty}(\mathbb{T}_{\theta}^2)$.
 We consider a topologically trivial line bundle equipped with a general holomorphic structure and the corresponding twisted Dolbeault Laplacians.
We define an spectral triple ($A_{\theta}, \mathcal{H}, D)$  that encodes the twisted Dolbeault complex of  $ A_{\theta}$ and whose index gives the left
 hand side of the  Riemann-Roch formula.
 Using Connes' pseudodifferential calculus  and heat equation techniques, we  explicitly compute  the $b_2$ terms of the asymptotic expansion of
$\text{Tr} (e^{-tD^2})$. We find that the curvature term  on the right hand side of  the Riemann-Roch formula  coincides with the scalar curvature
of the noncommutative torus recently defined and computed in \cite{CM1} and \cite{FK2}.
\end{abstract}

\begin{center}\tableofcontents
\end{center}

\section{Introduction}
Let $M$ be a  closed,  connected, and oriented surface equipped with a Riemannian metric and hence with a canonically defined complex structure. The complex structure is invariant under   conformal perturbations   of the metric.   Let $E$  be a holomorphic  line   bundle on $M$ equipped with a hermitian metric,  and let $ \delb_E:  C^{\infty}(M, E) \to \Omega^{(0, 1)} (M, E)$  denote the unique holomorphic connection on $E$ compatible with its hermitian metric. It is an elliptic operator and the Riemann-Roch formula computes  its index as follows.  Let $\Delta_0 = \delb_E^* \delb_E$
and   $\Delta_1= \delb_E \delb_E^*$ denote the corresponding Laplacians acting on smooth sections of $E$ and smooth $(0, 1)$-forms with coefficients in $E$, respectively. We have asymptotic expansion for heat traces
$$\text{Tr} \, e^{-t \Delta} \sim t^{-1}\sum_{n=0}^{\infty}(\int_M b_{2n} (x, \Delta) \,dV)\,t^{n} \quad \quad (t \to 0^+),$$
where $\Delta =\Delta_i, i=0,1,$  and $dV$ is the volume form of $M$.  Then we have
\begin{align}\label{index}
\text{index}\, (\delb_E)= \int_M (b_2 (x, \Delta_0) -b_2 (x, \Delta_1))\, dV =  \frac{1}{4 \pi}\int_M R+\frac{1}{2 \pi i} \int_M R_E,
\end{align}
where $R= K dV$ is the Gaussian curvature form of  $M$,   and $R_E$  is the curvature 2-form of $E$ \cite{G}.

In this paper we prove the analogue of the Riemann-Roch formula  for  the noncommutative two torus $ A_{\theta} = C(\mathbb{T}_{\theta}^2)$ equipped
with an arbitrary translation invariant complex structure and Weyl factor.  For the line bundle we consider a topologically trivial line bundle equipped
with a general, non-trivial,  holomorphic structure. Extending the spectral triple of \cite{CT}, we define an spectral triple ($A_{\theta}, \mathcal{H}, D)$  that encodes
the twisted Dolbeault complex of  $ A_{\theta}$ and whose index gives the left hand side of the  Riemann-Roch formula.
 Using Connes' pseudodifferential calculus \cite{CT} and heat equation techniques, we  explicitly compute  the $b_2$ terms of the asymptotic expansion
of $\text{Tr} (e^{-tD^2})$. We find that the $R$-term on the right hand side of  the Riemann-Roch formula  coincides with the scalar curvature of the
noncommutative torus recently defined and computed in \cite{CM1} and \cite{FK2}. We also show that the  topological term   $R_E$, the Chern form,
vanishes.
For the trivial holomorphic structure, we recover the Gauss-Bonnet theorem of   \cite{CT} and  \cite{FK}. In the last section we consider the nontrivial projective module $\mathcal S(\mathbb R)$ equipped with a holomorphic structure and by a variational argument verify the statement of the Riemann-Roch theorem for the twisted Dolbeault operator.

 Following the pioneering work of Connes and Tretkoff  on the Gauss-Bonnet theorem  for the noncommutative two torus \cite{CT},
and its extension and refinement in \cite{FK},  the question of defining and computing the scalar curvature for the noncommutative two
torus was  eventually settled in the paper of Connes and Moscovici in \cite{CM1}, and,  independently,
 by Fathizadeh and Khalkhali in \cite{FK2}. So the question of a Riemann-Roch formula for $ A_{\theta}$ posed
itself in a natural way at this stage. M. K.   would like to thank Farzad Fathizadeh for continued collaboration and many informative discussions.

\section{Conformal structures on the irrational rotation algebra.}\label{conf}

For an irrational number $\theta$, the $C^*$-algebra
 $A_{\theta}$ is, by definition, the universal unital $C^*$-algebra  generated by two unitaries $U, V$ satisfying
 $$VU=e^{2 \pi i \theta} UV.$$
There is a continuous action of $\mathbb{T}^2$, $\mathbb{T}$, on $A_{\theta}$ given by 
\[\alpha_s(U^mV^n)=e^{is.(m,n)}U^mV^n,\]
and the space of smooth elements of $\At$ under this action will be denoted by $A_{\theta}^{\infty}$. This algebra is also can be described as
\[
A_{\theta}^{\infty}=\big \{\sum_{m,n\in \mathbb{Z}}a_{m,n}U^mV^n; \quad  (1+|m|^k+|n|^q)|a_{m,n}| < \infty,
\forall k,q \in \mathbb{Z}  \big \}.
\]
There is a unique normalized trace   $\tau_0$  on $A_{\theta}$ that on smooth elements is given by
\[\tau_0\,(\sum_{m,n\in \mathbb{Z}}a_{m,n}U^mV^n)=a_{0,0}.\]

There are two derivations denoted by $\delta_1, \, \delta_2: A_{\theta}^{\infty} \to A_{\theta}^{\infty}$ induced by the action of $\mathbb{T}^2$ on $A_{\theta}$. On the generators they are defined by
\[\delta_1(U)=U, \quad \delta_1(V)=0, \quad  \delta_2(U)=0, \quad  \delta_2(V)=V.\]
These derivations anti-commute with the $*$-operator of $\At$, i.e. one has $\delta_j(a^*)= -\delta_j(a)^* $ for $j=1, 2$ and $a\in A_{\theta}^{\infty}$ and also they are invariant under the trace.
 \[\tau_0\, \circ \delta_j =0, \quad\text{for}\quad j=1, 2.\]
This yields
\[ \tau_0\,(a\delta_j(b)) = -\tau_0\,(\delta_j(a)b), \qquad \forall a,b \in A_{\theta}^{\infty}. \]

There exists an inner product on $\At$ given by
\[ \langle a, b \rangle = \tau_0\,(b^*a), \qquad a,b \in A_{\theta}. \nonumber \]
The Hilbert space completion of $\At$ under this inner product will be denoted by $\mathcal{H}_0$. The derivations $\delta_1, \delta_2$, as unbounded
operators on $\mathcal{H}_0$,  are formally selfadjoint and have unique extensions to selfadjoint operators.

For any complex number $\tau$ in the upper half plane, there exists a complex structure on the noncommutative two torus given by
\[ \partial = \delta_1 + \bar \tau \delta_2, \qquad \partial^*=  \delta_1 + \tau \delta_2. \]
The associated positive Hochschild  two cocycle on  $A_{\theta}^{\infty}$ is given by (\emph{cf.}
\cite{CT})
\[ \psi (a, b, c)=- \tau_0\, (a\partial b \partial^* c).\]

The space of $(1,0)$-forms in this case is defined by the completion of the space of finite sums $\sum a \partial b$, $a,b
\in A_{\theta}^{\infty}$
with respect to the inner product defined by the above positive cocycle. Note that $\partial$ is an unbounded operator on  $\mathcal{H}_0$ and $\partial^*$ is
 its formal adjoint. 

One can change the metric $\tau_0$,  inside the conformal class of the metric \cite{CT}, see also \cite{CohC}, by choosing a smooth
selfadjoint element $h=h^* \in A_{\theta}^{\infty}$, and defining a linear functional
$\varphi$ on $A_{\theta}$ by
\[\varphi(a)= \tau_0\,(ae^{-h}), \qquad a \in A_{\theta}.\]
The map $\varphi$ is a positive linear functional which is a twisted trace and satisfies the KMS condition at $\beta = 1$ for the
1-parameter group $\{\sigma_t \}$, $t \in \mathbb{R}$ of inner automorphisms 
\[
\sigma_t(x)=e^{ith}xe^{-ith}.
\]
We have $\sigma_t= \Delta^{-it}$ where the modular operator for $\varphi$ is 
(\emph{cf.} \cite{CT})
\[\Delta(x)=e^{-h}xe^{h}.\]
The 1-parameter group of automorphisms $\sigma_t$ is generated by the derivation
$- \log \Delta$ where
\[
\log \Delta (x) = [-h,x], \qquad x \in A_{\theta}^{\infty}.
\]

We  define an inner product $\langle \, , \, \rangle_{\varphi}$ on $A_{\theta}$ by
\[ \langle a,b  \rangle_{\varphi} = \varphi(b^*a), \qquad a,b \in A_{\theta}. \]
The Hilbert space obtained from completing $A_{\theta}$ with respect to this
inner product will be denoted by $\mathcal{H}_{\varphi}$.

\section{Topologically  trivial bundles with arbitrary holomorphic structures}

It is well known that holomorphic structures on a  trivial  line bundle over  a compact Riemann surface are parameterized by points
of  the Jacobian of the surface.  Thus for genus one surfaces  they are in one to one correspondence with points of the surface itself.
Its noncommutative analogue is as follows.  For a  noncommutative two torus, a
holomorphic structure   on  $ \mathcal E=\At$, considered as a free  $\At^{\infty}$-module,  is  given by a  holomorphic flat connection
\[
\nabla = \del+w: \mathcal E \to \Omega^{(1,0)}\otimes_{\Ats}\mathcal E,
\]
where $w\in \C$ and
\[\del=\delta_1+\overline\tau \delta_2.
\]
Considered as a densely defined unbounded operator $\nabla: \mathcal H_0 \to \mathcal{H}^{(1,0)}$ has a formal adjoint  given by
\[
\nabla^*=\overline  w+\del^*,
\]
where $\del^*=
\delta_1+\tau\delta_2$. Note that for the trivial bundle $\mathcal E$, the completion of $\Omega^{(1,0)}\otimes_{\Ats} \mathcal E$ can be identified by $\mathcal H^{(1,0)}$.
The Laplacian on (0,0)-sections is given by
\[\Delta_0=
\nabla^*\nabla=(\overline w+\del^*)( w+\del)=| w|^2+\overline w \del+w\del^*+\del^*\del.
\]

Let us view the operator $\nabla$ as an unbounded operator from $\mathcal{H}_{\varphi}$ to
 $\mathcal{H}^{(1,0)}$
and denote it by $\nabla_{\varphi}$. Similar to \cite{CT}, we construct an \emph{even spectral triple} by considering the
left action of $A_{\theta}$ on the
Hilbert space
\[
 \mathcal{H}= \mathcal{H}_{\varphi}\oplus \mathcal{H}^{(1,0)},
\]
and the operator
\[
 D=
\left(\begin{array}{c c}
0 & \nabla^* _{\varphi}\\
\nabla_{\varphi} & 0
\end{array}\right)
: \mathcal{H} \to \mathcal H.\]

Then  the \emph{Laplacian} has the following form:
\begin{equation}
\triangle := D^2 =  \left(\begin{array}{c c}
 \nabla^*_{\varphi}\nabla_{\varphi} & 0  \\
0 &  \nabla_{\varphi} \,\nabla^*_{\varphi}
\end{array}\right), \nonumber
 \end{equation}
 and  the \emph{grading} is given by
 \[\gamma =  \left(\begin{array}{c c}
 1 & 0  \\
0 & -1
\end{array}\right) : \mathcal{H} \to \mathcal{H}.
\]

A twisted spectral triple is constructed over $A_\theta^{op}$ in \cite{CM1}, using the Tomita anti-linear unitary map $J_\varphi$
in $\mathcal{H}_\varphi$, and the unitary right action of $A_\theta$ in $\mathcal{H}_\varphi$ given by
$a \mapsto J_\varphi a^* J_\varphi$.
It is shown that $(A_\theta^{\textnormal{op}}, \mathcal{H}, D)$ is a twisted spectral triple, see \cite{CM1}.

One can show that changing the metric within the conformal class of a given metric gives a new Laplacian  $\Delta'_{0}$ and one has
\begin{lemma}The operator $\Delta'_0: \mathcal H_0\to \mathcal H_0$ is anti-unitarily equivalent to the operator
$k\widetilde\Delta_0 k$, where
\[\widetilde\Delta_0=(-\overline w+\del^*)(- w+\del).
\]
\end{lemma}
\begin{proof}
First note that the map $W:\mathcal H_0\to \mathcal H_{\varphi}$, given by $W(a)=ak$ is an isometry.
One has $\nabla_{\varphi}\circ W=\nabla\circ R_k$ and hence
\[
W^*\nabla_{\varphi}^*\nabla_{\varphi} W=R_k\nabla^*\nabla R_k.
\]

Therefore the operators $\Delta_0'$ and $R_k \Delta_0 R_k$ are unitary equivalent. On the other hand
\[
J R_k(\overline w+\del^*)( w+\del) R_kJ= J  R_k J J(\overline w+\del^*)( w+\del) J J R_kJ= k(-\overline w+\del^*)(- w+\del)k.
\]
\end{proof}
On the other hand on the space of twisted $(1,0)$-sections the twisted Dolbeault Laplacian is given by
\[\Delta_1=
\nabla\nabla^*=( w+\del)(\overline w+\del^*)=| w|^2+\overline w \del+w\del^*+\del\del^*.
\]
Perturbing the metric on its conformal class yields a new Laplacian $\Delta'_{1}$ and we have
\begin{lemma}The operator $\Delta'_1: \mathcal H^{(1,0)}\to \mathcal H^{(1,0)}$ is anti-unitarily equivalent to the operator

\[
(-\overline w+\del^*)k^2(- w+\del)
\]
\end{lemma}
\begin{proof}
The proof is similar to the previous lemma.
\end{proof}
\subsection{Connes' pseudodifferential operators on $\mathbb{T}_\theta^2$.}
For a non-negative integer $n$, the space of differential operators
on $A_{\theta}^{\infty}$ of order at most $n$ is defined to be the
vector space of operators of the form
\[ \sum_{j_1+j_2 \leq n } a_{j_1,j_2}
\nabla_1^{j_1} \nabla_2^{j_2}, \qquad j_1, \, j_2 \geq 0, \qquad a_{j_1,j_2}
\in A_{\theta}^{\infty},\]
where $\nabla_i=\delta_i+z_i$ for $i=1,2$. Here $z_1$ and $z_2$ are two complex numbers such that $w=z_1+\overline \tau z_2$.

The notion of a differential operator on $A_{\theta}^{\infty}$ can be generalized
to the notion of a pseudodifferential operator using operator valued symbols \cite{C0}. In fact
this is achieved by considering the pseudodifferential calculus associated to $C^*$-dynamical systems
\cite{C0}, for the canonical dynamical system $(A_\theta^\infty, \{\alpha_s \})$. In the
sequel, we shall use the notation $\partial_1=\frac{\partial}{\partial \xi_1}$, $\partial_2=\frac{\partial}{\partial \xi_2}$.

\begin{definition}
For an integer $n$, a smooth map $\rho: \mathbb{R}^2 \to A_{\theta}^{\infty}$ is  said
to be a symbol of order $n$, if for all non-negative integers $i_1, i_2, j_1,
j_2,$
\[ ||\delta_1^{i_1} \delta_2^{i_2} \partial_1^{j_1} \partial_2^{j_2} \rho(\xi) ||
\leq c (1+|\xi|)^{n-j_1-j_2},\]
where $c$ is a constant, and if there exists a smooth map $k: \mathbb{R}^2 \to
A_{\theta}^{\infty}$ such that
\[\lim_{ \lambda \to \infty}  \lambda^{-n} \rho( \lambda\xi_1,  \lambda\xi_2) = k (\xi_1, \xi_2).\]
The space of symbols of order $n$ is denoted by $S_n$.
\end{definition}

To a symbol $\rho$ of order $n$, one can associate an operator on $A_{\theta}^{\infty}$,
denoted by $P_{\rho}$, given by
\[ P_{\rho}(a) = (2 \pi)^{-2} \int \int e^{-is \cdot \xi} \rho(\xi) \alpha_s(a) \,ds \,
d\xi. \]
The operator $P_{\rho}$ is said to be a pseudodifferential operator of order $n$. For
example, the differential operator $\sum_{j_1+j_2 \leq n } a_{j_1,j_2}
\delta_1^{j_1} \delta_2^{j_2}$ is associated with the symbol $\sum_{j_1+j_2 \leq n } a_{j_1,j_2}
\xi_1^{j_1} \xi_2^{j_2}$ via the above formula.

One can define the equivalent symbols as \cite{CT} and find the multiplication and adjoint symbol formula.

\begin{definition}\label{elliptic}
Let $\rho$ be a symbol of order $n$. It is said to be elliptic if $\rho(\xi)$ is invertible for $\xi \neq 0$, and if
there exists a constant $c$ such that
\[ || \rho(\xi)^{-1} || \leq c (1+|\xi|)^{-n} \]
for sufficiently large $|\xi|.$
\end{definition}

By the Cauchy integral formula, for $i=0,1$, one has
\[
e^{-t\Delta'_i}=\frac{1}{2\pi i}
\int_C e^{-t\lambda}(\Delta'_i-\lambda)^{-1} d\lambda,
\]
where $C$ is a curve in $\C$, which goes around the non-negative real axis in counter clockwise direction without touching it. From this one can obtain
the asymptotic expansion
\[
\text{Tr}(e^{-t\Delta'_i})\sim t^{-1}\sum B_{2n}(\Delta'_i) t^n, \qquad t\to 0^+, \qquad i=0,1.
\]
By McKean-Singer formula one has the following formula for the index of the twisted Dolbeault complex
\[
\text{Index}(\nabla_{\varphi})=\sum_{i=0}^{1} (-1)^i \text{Tr}(e^{-t \Delta'_i}).
\]
This gives us the following formula for the index
\[\text{Index}(\nabla_{\varphi})=B_2(\Delta'_0)-B_2(\Delta'_1).
\]
To find this value, one can approximate the inverse of $(\Delta'_i-\lambda)$ by a pseudodifferential operator
$B_{\lambda}$ with a symbol $\sigma(B_{\lambda})$ of the form
\[
b_0^i(\xi,\lambda)+b_1^i(\xi,\lambda)+b_2^i(\xi,\lambda)+\cdots,
\]
where $b_j^i(\xi,\lambda)$ is a symbol of order $-2-j$ for $i=0,1$, and
\[
\sigma(B_{\lambda}(\Delta'_i-\lambda))\sim1.
\]
Therefore, one can find that
\[
B_2(\Delta'_i)=\frac{1}{2\pi i}\int \int_C e^{-\lambda}\tau_0(b_2^i(\xi,\lambda))d \xi d\lambda.
\]
As in \cite{CT,CM1}, one can see that the contour integration can be dropped by a homogeneity argument and therefore
\[
B_2(\Delta'_i)=\int \tau_0(b_2^i(\xi,-1))d \xi, \qquad i=0,1.
\]
In the next section we will compute this index and show that it is zero.

\section{The computation of $B_2(\Delta'_0)$ and $B_2(\Delta'_1)$.}

In order to find the value of the $B_2(\Delta'_i)$ for $\Delta'_0\sim k(-\overline w+\del^*)(- w+\del)k$
and $ \Delta'_1\sim(-\overline w+\del^*)k^2(- w+\del) $, following the approach described in
 the previous section, we need to find the symbol of these operators.

For $w=0$, the following lemma reduce to Lemma 4.1 of \cite{FK}.
\begin{lemma}
The symbol of the operator $\sigma(k\widetilde\Delta_0 k)=a_0(\xi)+a_1(\xi)+a_2(\xi)$
\[
a_2=\xi_1^2k^2+|\tau|^2\xi_2^2k^2+2\tau_1 \xi_1\xi_2k^2,\]
\[
a_1=2\xi _1k \delta _1(k)+2|\tau |^2\xi _2k \delta _2(k)+2\tau _1 \xi _1k \delta _2(k)
+2\tau _1 \xi _2k \delta _1(k)-2w_1 \xi _1k^2-2w_1\tau _1  \xi _2k^2+2w_2 \tau _2 \xi _2k^2 \]
\[
a_0=k \delta _1 \delta _1(k) +|\tau |^2k \delta _2 \delta _2(k) +2\tau _1 k \delta _1 \delta _2(k)
-2w_1k \delta _1(k)-2w_1\tau _1k \delta _2(k)+2w_2\tau _2k \delta _2(k)+|w|^2k^2
\]
\end{lemma}

\begin{proof}
The proof easily can be obtained from the fact that
$\sigma(- w+\del)=-w+\xi_1+\overline \tau \xi_2$, $\sigma(-\overline w+\del^*)=-\overline w+\xi_1+ \tau \xi_2$.
\end{proof}

\begin{lemma}
The symbol of the operator $\sigma((-\overline w+\del^*)k^2(- w+\del))=a_0(\xi)+a_1(\xi)+a_2(\xi)$,
\[
a_2=\xi_1^2k^2+|\tau|^2\xi_2^2k^2+2\tau_1 \xi_1\xi_2k^2,\]
\[a_1=\xi _1\delta _1(k^2)+\tau  \xi _1\delta _2(k^2)+\overline\tau\xi _2\delta _1(k^2)+|\tau |^2\xi _2\delta _2(k^2)-2w_1\xi _1k^2-2w_1\tau _1\xi _2k^2+2\tau _2w_2\xi _2k^2\]
\[
a_0=-w\delta _1(k^2)-w\tau \delta _2(k^2)+|w|^2k^2
\]
\end{lemma}

\begin{proof} 
The proof is similar to the previous lemma.
\end{proof}

Using similar methods as in \cite{CT,FK}, one can obtain the folowing formulas for $b_2$:

\begin{align*}
b_2=-(b_0a_0b_0+b_1a_1b_0+\del_i(b_0)\delta_i(a_1)b_0+\del_i(b_1)\delta_i(a_2)b_0+1/2 \del_i\del_j(b_0) \delta_i\delta_j(a_2)b_0).
\end{align*}
Here we have used the summation on repeated indices.

\subsection{The computation of $B_2(\Delta'_0)$ }
 One can apply the method of \cite{CT,FK,FK2} to obtain all the terms. We just give the terms which are new compared to those which already appeared in  \cite{FK2} (after integration over the $\xi$ plane). The rest of therms are as those found in \cite{FK2}.

To do integration over $\xi$, one uses the following substitution rule (see \cite{FK})
\begin{align}\label{GC formula}
\xi_1=r\cos \theta -\frac{\tau_1}{\tau_2}r\sin\theta,\qquad \xi_2=\frac{r}{\tau_2}\sin\theta.
\end{align}
Up to an overall factor $\frac{r}{\tau_2}$, the extra terms comparing $\cite{FK2}$ are
\begin{align*}
&-2 \pi  |w|^2 b_0^2 k^2+4 \pi  r^2| w|^2 b_0^3 k^4
+4 \pi  w_1 b_0^2 k \delta _1(k)-2 \pi  b_0^2 k \delta _1 \delta _1(k) -4 \pi  \tau _1 b_0^2 k \delta _1 \delta _2(k)
\\&+4 \pi  w_1 \tau _1 b_0^2 k \delta _2(k)
-4 \pi  w_2 \tau _2 b_0^2 k \delta _2(k)-2 \pi  \tau _1^2 b_0^2 k \delta _2 \delta _2(k) -2 \pi  \tau _2^2 b_0^2 k \delta _2 \delta _2(k)
+8 \pi  r^2 b_0^3 k^2 \delta _1(k) ^2
\\&
+8 \pi  r^2 \tau _1^2 b_0^3 k^2 \delta _2(k) ^2+8 \pi  r^2 \tau _2^2 b_0^3 k^2 \delta _2(k) ^2-24 \pi  r^2 w_1 b_0^3 k^3 \delta _1(k)
+8 \pi  r^2 b_0^3 k^3 \delta _1 \delta _1(k)
\\&
+12 \pi  r^2 \tau _1 b_0^3 k^3 \delta _1 \delta _2(k) -24 \pi  r^2 w_1 \tau _1 b_0^3 k^3 \delta _2(k)+24 \pi  r^2 w_2 \tau _2 b_0^3 k^3 \delta _2(k)
\\&
+4 \pi  r^2 \tau _1 b_0^3 k^3 \delta _2 \delta _1(k) +8 \pi  r^2 \tau _1^2 b_0^3 k^3 \delta _2 \delta _2(k) +8 \pi  r^2 \tau _2^2 b_0^3 k^3 \delta _2 \delta _2(k)
-8 \pi  r^4 b_0^4 k^4 \delta _1(k) ^2
\\&
-8 \pi  r^4 \tau _1^2 b_0^4 k^4 \delta _2(k) ^2-8 \pi  r^4 \tau _2^2 b_0^4 k^4 \delta _2(k) ^2+24 \pi  r^4 w_1 b_0^4 k^5 \delta _1(k)
-8 \pi  r^4 b_0^4 k^5 \delta _1 \delta _1(k)
\\& -16 \pi  r^4 \tau _1 b_0^4 k^5 \delta _1 \delta _2(k) +24 \pi  r^4 w_1 \tau _1 b_0^4 k^5 \delta _2(k)
-24 \pi  r^4 w_2 \tau _2 b_0^4 k^5 \delta _2(k)-8 \pi  r^4 \tau _1^2 b_0^4 k^5 \delta _2 \delta _2(k)
\\& -8 \pi  r^4 \tau _2^2 b_0^4 k^5 \delta _2 \delta _2(k)
+8 \pi  r^2 \tau _1 b_0^3 k^2 \delta _1(k) \delta _2(k)+8 \pi  r^2 \tau _1 b_0^3 k^2 \delta _2(k) \delta _1(k)-8 \pi  r^4 \tau _1 b_0^4 k^4 \delta _1(k) \delta _2(k)\\&
-8 \pi  r^4 \tau _1 b_0^4 k^5 \delta _2(k) \delta _1(k).
\end{align*}

Integrating over $r$, up to factor $\frac{\pi}{3\tau_2}$,  gives us
\begin{align*}
&-k^{-1}\delta_1\delta_1(k)-|\tau|^2k^{-1}\delta_2\delta_2(k)-2k^{-1}\delta_1\delta_2(k)\\
&+2k^{-2}\delta_1(k)^2+2|\tau|^2k^{-2}\delta_2(k)^2+2\tau_1(\delta_1(k)\delta_2(k)+\delta_2(k)\delta_1(k)).
\end{align*}
Surprisingly, the holomorphic structure of the bundle does not contribute in the formula above and this exactly coincides with the terms found in Section 4.1. in \cite{FK}. Indeed, we have

\[\int_0^{\infty}\Big(-2 \pi  |w|^2 b_0^2 k^2+4 \pi  r^2 |w|^2 b_0^3 k^4\Big)(\frac{rdr}{\tau_2})=0,
\]
\[
\int_0^{\infty}\Big(-4 \pi  w_1 b_0^2 k \delta _1 (k)+24 \pi  r^4 w_1 b_0^4 k^5 \delta _1 (k)-24 \pi  r^2 w_1 b_0^3 k^3 \delta _1 (k)\Big)(\frac{rdr}{\tau_2})=0,
\]
and
\begin{align*}
\int_0^{\infty}\Big(24 \pi  r^4 w_1 \tau _1 b_0^4k^5-24 \pi  r^4 w_2 \tau _2 b_0^4k^5-24 \pi  r^2 w_1 \tau _1 b_0^3k^3+24 \pi  r^2 w_2 \tau _2 b_0^3k^3\\
+4 \pi  w_1 \tau _1 b_0^2k-4 \pi  w_2 \tau _2 b_0^2k\Big)(\frac{rdr}{\tau_2})=0.
\end{align*}

\subsection{The computation of $B_2(\Delta'_1)$.}

Now we would like to give the computation of $B_2(\Delta'_1) $ on $(1,0 )$- sections. As in the previous section, we just mention the integrated (over $\xi$ plane of) the new terms.

 Up to overall factor of $r/\tau_2$ these terms are given by

\begin{align*}
&-2 \pi  |w|^2 b_0^2 k^2+4 \pi  r^2| w|^2 b_0^3 k^4+4 \pi  w_1 b_0^2 k \delta _1(k)+4 i \pi  w_2 b_0^2 k \delta _1(k)+4 \pi  w_1 \tau _1 b_0^2 k \delta _2(k)
\\&+4 i \pi  w_2 \tau _1 b_0^2 k \delta _2(k)+4 i \pi  w_1 \tau _2 b_0^2 k \delta _2(k)-4 \pi  w_2 \tau _2 b_0^2 k \delta _2(k)+8 \pi  r^2 b_0^3 k^2 \delta _1(k) ^2
\\&+8 \pi  r^2 \tau _1^2 b_0^3 k^2 \delta _2(k) ^2
+8 \pi  r^2 \tau _2^2 b_0^3 k^2 \delta _2(k) ^2-24 \pi  r^2 w_1 b_0^3 k^3 \delta _1(k)-8 i \pi  r^2 w_2 b_0^3 k^3 \delta _1(k)
\\
&+8 \pi  r^2 b_0^3 k^3 \delta _1 \delta _1(k) +12 \pi  r^2 \tau _1 b_0^3 k^3 \delta _1 \delta _2(k) +4 i \pi  r^2 \tau _2 b_0^3 k^3 \delta _1 \delta _2(k)
\\
&-24 \pi  r^2 w_1 \tau _1 b_0^3 k^3 \delta _2(k)-8 i \pi  r^2 w_2 \tau _1 b_0^3 k^3 \delta _2(k)-8 i \pi  r^2 w_1 \tau _2 b_0^3 k^3 \delta _2(k)\\
&+24 \pi  r^2 w_2 \tau _2 b_0^3 k^3 \delta _2(k)+4 \pi  r^2 \tau _1 b_0^3 k^3 \delta _2 \delta _1(k) -4 i \pi  r^2 \tau _2 b_0^3 k^3 \delta _2 \delta _1(k)
\\&
+8 \pi  r^2 \tau _1^2 b_0^3 k^3 \delta _2 \delta _2(k) +8 \pi  r^2 \tau _2^2 b_0^3 k^3 \delta _2 \delta _2(k) -8 \pi  r^4 b_0^4 k^4 \delta _1(k) ^2-8 \pi  r^4 \tau _1^2 b_0^4 k^4 \delta _2(k) ^2
\\&-8 \pi  r^4 \tau _2^2 b_0^4 k^4 \delta _2(k) ^2
+24 \pi  r^4 w_1 b_0^4 k^5 \delta _1(k)-8 \pi  r^4 b_0^4 k^5 \delta _1 \delta _1(k) -16 \pi  r^4 \tau _1 b_0^4 k^5 \delta _1 \delta _2(k)
\\&
+24 \pi  r^4 w_1 \tau _1 b_0^4 k^5 \delta _2(k)-24 \pi  r^4 w_2 \tau _2 b_0^4 k^5 \delta _2(k)-8 \pi  r^4 \tau _1^2 b_0^4 k^5 \delta _2 \delta _2(k)
\\&-8 \pi  r^4 \tau _2^2 b_0^4 k^5 \delta _2 \delta _2(k) +8 \pi  r^2 \tau _1 b_0^3 k^2 \delta _1(k) \delta _2(k)+8 \pi  r^2 \tau _1 b_0^3 k^2 \delta _2(k) \delta _1(k)
\\&
-8 \pi  r^4 \tau _1 b_0^4 k^4 \delta _1(k) \delta _2(k)-8 \pi  r^4 \tau _1 b_0^4 k^4 \delta _2(k) \delta _1(k).
\end{align*}

Integrating over $r$, up to a factor of $\frac{2 \pi }{3  \tau _2}$, gives
\begin{align*}
&k^{-2}\delta_1(k)^2+|\tau|^2k^{-2}\delta_2(k)^2+k^{-1}\delta_1\delta_1(k)+|\tau|^2k^{-1}\delta_2\delta_2(k)\\&
+\tau_1k^{-1}\delta_1\delta_2(k)+\tau_1k^{-1}\delta_2\delta_1(k)+\tau_1k^{-2}\delta_1(k)\delta_2(k)+\tau_1\delta_2(k)\delta_1(k).
\end{align*}
Again, the holomorphic structure of the bundle does not contribute to the formula and this is because
\[
\int_0^{\infty}\Big(-2 \pi  |w|^2 b_0^2k^2+4 \pi  r^2| w|^2 b_0^3k^4\Big)\frac{rdr}{\tau_2}=0,
\]
\[
\int_0^{\infty}\Big(24 \pi  r^4 w_1 b_0^4k^5-24 \pi  r^2 w_1 b_0^3k^3-8 i \pi  r^2 w_2 b_0^3k^3+4 \pi  w_1 b_0^2k+4 i \pi  w_2 b_0^2k\Big)\frac{rdr}{\tau_2}=0,
\]
and
\begin{align*}
\int_0^{\infty}\Big(24 \pi  r^4 w_1 \tau _1 b_0^4k^5-24 \pi  r^4 w_2 \tau _2 b_0^4k^5-24 \pi  r^2 w_1 \tau _1 b_0^3k^3-8 i \pi  r^2 w_2 \tau _1 b_0^3k^3-
8 i \pi  r^2 w_1 \tau _2 b_0^3k^3\\+24 \pi  r^2 w_2 \tau _2 b_0^3k^3+
4 \pi  w_1 \tau _1 b_0^2k+4 i \pi  w_2 \tau _1 b_0^2k+4 i \pi  w_1 \tau _2 b_0^2k-4 \pi  w_2 \tau _2 b_0^k\Big)\frac{rdr}{\tau_2}=0.
\end{align*}

One can combine the terms with $b_0^2$ in the middle with the terms with $b_0$ in the middle. First for the sake of notation we introduce
\[
g^{ij}=\begin{pmatrix}
1 & \overline \tau\\
\tau & |\tau|^2
\end{pmatrix}
,\quad h^{ij}=\begin{pmatrix}
1 &  \tau_1\\
\tau_1 & |\tau|^2
\end{pmatrix}.
\]

Now combining these two type terms one has
\begin{align*}
T=&-2\pi g^{ij} r^3 b_0^2k^2\delta_i(k)b_0\delta_j(k)-4\pi g^{ij} r^3 b_0^2k\delta_i(k)b_0k\delta_j(k)-2\pi g^{ij} r^3 b_0^2\delta_i(k)b_0k^2\delta_j(k)
\\
&
+4\pi g^{ij} r^5 b_0^3k^4\delta_i(k)b_0\delta_j(k)+8\pi g^{ij} r^5 b_0^3k^3\delta_i(k)b_0k\delta_j(k)+4\pi g^{ij} r^5 b_0^3k^2\delta_i(k)b_0k^2\delta_j(k)
\\
&
-4\pi h^{ij} r^7 b_0^4k^6\delta_i(k)b_0\delta_j(k)-8\pi h^{ij} r^7 b_0^4k^5\delta_i(k)b_0k\delta_j(k)
-4\pi h^{ij} r^7 b_0^4k^4\delta_i(k)b_0k^2\delta_j(k)
\end{align*}

Here we have used the summation on repeated indices. Applying the computation given in \cite{FK} page 20, for the first term of $T$, we get
\[
\int -2\pi g^{ij} r^3 b_0^2k^2\delta_i(k)b_0\delta_j(k) dr=-\pi g^{ij}k^{-2}\mathcal D_1(\delta_i(k))\delta_j(k),
\]
where $\mathcal D_m=\mathcal L_m(\Delta)$ and $\mathcal L_m$ is the modified logarithm function given by \cite{CT}
\[
\mathcal L_m(u)=(-1)^m(u-1)^{-(m+1)}\Big (\log u -\sum_{i=1}^m (-1)^{i+1}\frac{(u-1)^i}{i}\Big ).
\]
For the second and third terms of first line of $T$, we use Connes-Tretkoff lemma \cite{CT} to get
as
\[
\int -4\pi g^{ij} r^3 b_0^2k\delta_i(k)b_0k\delta_j(k)dr=-2\pi g^{ij}k^{-2}\mathcal D_1\Delta^{1/2}(\delta_i(k))\delta_j(k),
\]
and
\[
\int -2\pi g^{ij} r^3 b_0^2\delta_i(k)b_0k^2\delta_j(k)dr=-\pi g^{ij}k^{-2}\mathcal D_1\Delta(\delta_i(k))\delta_j(k).
\]
For the second line in $T$ one has
\begin{align*}
\int 4 \pi g^{ij} r^5 b_0^3k^4\delta_i(k)b_0\delta_j(k)dr&=2\pi g^{ij}k^{-2}\mathcal D_2(\delta_i(k))\delta_j(k)
\end{align*}
\[
\int 8\pi g^{ij} r^5 b_0^3k^3\delta_i(k)b_0k\delta_j(k)dr=4\pi g^{ij}k^{-2}\mathcal D_2\Delta^{1/2}(\delta_i(k))\delta_j(k),
\]
and
\[
\int 4\pi g^{ij} r^5 b_0^3k^2\delta_i(k)b_0k^2\delta_j(k)dr=2\pi g^{ij}k^{-2}\mathcal D_2\Delta(\delta_i(k))\delta_j(k).
\]
For the last line of $T$, we have
\[
\int-4\pi h^{ij} r^7 b_0^4k^6\delta_i(k)b_0\delta_j(k)dr=-2\pi h^{ij}k^{-2}\mathcal D_3(\delta_i(k))\delta_j(k),
\]
\[
\int-8\pi h^{ij} r^7 b_0^4k^5\delta_i(k)b_0k\delta_j(k)dr=-4\pi h^{ij}k^{-2}\mathcal D_3\Delta^{1/2}(\delta_i(k))\delta_j(k),
\]
\[
\int -4\pi h^{ij} r^7 b_0^4k^4\delta_i(k)b_0k^2\delta_j(k)dr=-2\pi h^{ij}k^{-2}\mathcal D_3\Delta(\delta_i(k))\delta_j(k).
\]
Hence

\begin{align*}
B_2(\Delta'_1)=
&\frac{2\pi}{3}\Big(k^{-2}\delta_1(k)^2+|\tau|^2k^{-2}\delta_2(k)^2+k^{-1}\delta_1\delta_1(k)+|\tau|^2k^{-1}\delta_2\delta_2(k)
\\&
+2\tau_1k^{-1}\delta_1\delta_2(k)+\tau_1k^{-2}\delta_1(k)\delta_2(k)+\tau_1k^{-2}\delta_2(k)\delta_1(k)\Big )
\\
&
-\pi g^{ij}k^{-2}\mathcal D_1(\delta_i(k))\delta_j(k)-2\pi g^{ij}k^{-2}\mathcal D_1\Delta^{1/2}(\delta_i(k))\delta_j(k)
-\pi g^{ij}k^{-2}\mathcal D_1\Delta(\delta_i(k))\delta_j(k)\\
&+
2\pi g^{ij}k^{-2}\mathcal D_2(\delta_i(k))\delta_j(k)+4\pi g^{ij}k^{-2}\mathcal D_2\Delta^{1/2}(\delta_i(k))\delta_j(k)
+2\pi g^{ij}k^{-2}\mathcal D_2\Delta(\delta_i(k))\delta_j(k)\\
&-2\pi h^{ij}k^{-2}\mathcal D_3(\delta_i(k))\delta_j(k)-4\pi h^{ij}k^{-2}\mathcal D_3\Delta^{1/2}(\delta_i(k))\delta_j(k)-2\pi h^{ij}k^{-2}\mathcal D_3\Delta(\delta_i(k))\delta_j(k).
\end{align*}
Therefore
\begin{align*}
B_2(\Delta'_1)=&
2\pi k^{-2} h^{ij}\Big(\frac{1}{3}\delta_i(k)\delta_j(k)+\frac{1}{3}\Delta^{-1/2}(\delta_i(k))\delta_j(k)-\mathcal D_3(1+\Delta^{1/2})^2\delta_i(k)\delta_j(k)\Big)
\\
&+\pi k^{-2} g^{ij}\Big((-\mathcal D_1+2\mathcal D_2) ((1+\Delta^{1/2})^2(\delta_i(k))\delta_j(k)\Big ).
\end{align*}

To find a simpler formula, note that
\[
g^{ij}=h^{ij}+\epsilon^{ij},
\]
where
\[
\epsilon^{ij}:=i\tau_2\begin{pmatrix}
                        0&-1\\
                         1&0
                        \end{pmatrix}.
\]
Then

\begin{align*}
B_2(\Delta'_1)=&\pi k^{-2} h^{ij}\Big(\frac{2}{3}\delta_i(k)\delta_j(k)+\frac{2}{3}\Delta^{-1/2}(\delta_i(k))\delta_j(k)+\\
&\qquad(-\mathcal D_1+2\mathcal D_2-2\mathcal D_3)(1+\Delta^{1/2})^2\delta_i(k)\delta_j(k)\Big)
\\
+&\pi k^{-2} \epsilon^{ij}\Big((-\mathcal D_1+2\mathcal D_2) ((1+\Delta^{1/2})^2(\delta_i(k))\delta_j(k)\Big).
\end{align*}
Let us define

\[
f(u):=\frac{2}{3}+\frac{2}{3}u^{-1/2}+(-\mathcal L_1(u)+2\mathcal L_2(u)-2\mathcal L_3(u))(1+u^{1/2})^2,
\]
and
\[
g(u):=((-\mathcal L_1(u)+2\mathcal L_2(u)) (1+u^{1/2})^2.
\]
One can see that $f$ and $g$ can be written as $h(\log u)$ and $l(\log u)$, where
\[
h(x)=\frac{2 \text{Csch} \frac{x}{2} ^2 \left(3 x \text{Cosh}[x]+6 \text{Sinh} \frac{x}{2} -3 \text{Sinh}[x]-2 \text{Sinh} \frac{3 x}{2} \right)}{3 x^2},
\]
and
\[
l(x)=\frac{2+e^x (-2+x)+x}{\left(-1+e^{x/2}\right)^3 \left(1+e^{x/2}\right)}.
\]
The Taylor series of functions $h$ and $l$ up to order $9$ is given by
\[
h(x)=-\frac{x}{15}+\frac{x^2}{30}-\frac{5 x^3}{504}+\frac{11 x^4}{5040}-\frac{19 x^5}{67200}+\frac{x^6}{172800}+\frac{5 x^7}{2128896}+\frac{29 x^8}{106444800}+O[x]^{9}
\]
\[
l(x)=\frac{2}{3}-\frac{x}{3}+\frac{7 x^2}{120}-\frac{x^3}{720}-\frac{29 x^4}{40320}+\frac{x^5}{80640}+\frac{79 x^6}{4838400}-\frac{x^7}{9676800}-\frac{697 x^8}{1703116800}+O[x]^{9}
\]
\vskip .2in
Using the identities
\[
k^{-1}\delta_i(k)=2\frac{\Delta^{1/2}-1}{\Delta}\delta_i(\log k),\quad \delta_i(k)k^{-1}=-2\frac{\Delta^{-1/2}-1}{\Delta}\delta_i(\log k)
\]
and the fact that for an entire function $F$
\[
\tau_0(aF(\log \Delta)(b))=\tau_0(F(-\log \Delta )(a)b),
\]
one can find that
\[
\varphi(f(\Delta)(\delta_i(k))\delta_j(k))=\tau_0(K(\log k)(\delta_i(k))\delta_j(k))
\]
and
\[
\varphi(g(\Delta)(\delta_i(k))\delta_j(k))=\tau_0(K'(\log k)(\delta_i(k))\delta_j(k)),
\]
where $\varphi(x)=\tau_0(xk^{-2})$ and
\[
K(x)=4x^{-2}\left(e^{x/2}-1\right)^2h(x),\qquad L(x)=4x^{-2}\left(e^{x/2}-1\right)^2l(x).
\]

We would like to find the value of the
\[
B_2(\Delta'_1)=\pi k^{-2} h^{ij} \varphi(f(\Delta)(\delta_i(k))\delta_j(k))+\pi k^{-2} \epsilon^{ij} \varphi(g(\Delta)(\delta_i(k))\delta_j(k)).
\]
\newpage
The fact that $K$ is an odd function implies that
\[
\varphi(f(\Delta)(\delta_i(k))\delta_j(k))=-\varphi(f(\Delta)(\delta_j(k))\delta_i(k)).
\]
Hence
 \[\varphi(f(\Delta)(\delta_1(k))\delta_1(k))=\varphi(f(\Delta)(\delta_2(k))\delta_2(k))=0,\qquad
\]
and
\[
\varphi(f(\Delta)(\delta_1(k))\delta_2(k))=-\varphi(f(\Delta)(\delta_2(k))\delta_1(k)).
\]

On the other hand, for the second term, the fact that $\epsilon^{12}=-1,\,\epsilon^{21}=1$ together with $L$ being an even function implies that
\[
\pi k^{-2} \epsilon^{ij} \varphi(g(\Delta)(\delta_i(k))\delta_j(k))=0.
\]
Indeed
\begin{align*}
\varphi(g(\Delta)(\delta_i(k))\delta_j(k))=\varphi(g(\Delta)(\delta_j(k))\delta_i(k)).
\end{align*}
Now
\begin{align*}
\pi k^{-2} \epsilon^{ij} \varphi(g(\Delta)(\delta_i(k))\delta_j(k))=-\pi k^{-2}  \varphi(g(\Delta)(\delta_1(k))\delta_2(k))+\pi k^{-2} \varphi(g(\Delta)(\delta_2(k))\delta_1(k))=0.
\end{align*}
This completes the proof of the following:
\begin{theorem}
For any $w\in \C$, any  irrational number $\theta$, and any  positive invertible element $k\in A^{\infty}_{\theta}$, the value of $B_2(\Delta'_0)-B_2(\Delta'_1)$, where $\Delta'_0\sim k(-\overline w+\del^*)(- w+\del) k$ and $\Delta'_1\sim(-\overline w+\del^*)k^2(- w+\del)$,  is given by
\[
B_2(\Delta'_0)-B_2(\Delta'_1)=0
\]

\end{theorem}

\section{Computing the Riemann-Roch density}
In this section we find the analogue of the formula (\ref{index}), for the noncommutative two torus equipped with a holomorphic structure on its trivial bundle. Seeking a formula for the Riemann-Roch densities
$
\mathcal R-2i \mathcal R_E$ and $ \mathcal R^{\gamma}-2i\mathcal R_E^{\gamma}
$
on $(0,0)$ and $(1,0)$ sections,
we will need to work out the zeta functional $\zeta_i(a,s)=\text{Trace}(a\Delta_i^{-s})$ for $i=0,1$, \cite{C7,C5}. One has
\[
\text{Trace}(a\Delta_0^{-s})\big|_{s=0}+\text{Trace}(aP)=\tau_0(a( \mathcal R-2 i\mathcal R_E)),
\]
and
\[
\text{Trace}(a\Delta_1^{-s})\big|_{s=0}+\text{Trace}(aP)=\tau_0(a( \mathcal R^{\gamma}-2 i\mathcal R_E^{\gamma})),
\]
where $P$ is the projection on the kernel of $\Delta_i$ for $i=0,1$.
We would like to  show that $\mathcal R_E=\mathcal R_E^{\gamma}=0$,  $\mathcal R=R$ and $\mathcal R^{\gamma}=R^{\gamma}$, where $R$ and $R^{\gamma}$ are the scalar curvature and chiral scalar curvature of the noncommutative two torus introduced in \cite{CM1, FK2}. We will give the proof for $(1,0)$-sections, i.e. $\mathcal R^{\gamma}=R^{\gamma}$ and $\mathcal R_E^{\gamma}=0$, and the proofs for (0,0)-sections will be similar.
\begin{theorem}
The formulae for the Riemann-Roch densities are given by $\mathcal R_E=\mathcal R_E^{\gamma}=0$,  $\mathcal R=R$ and $\mathcal R^{\gamma}=R^{\gamma}$, where the (chiral) scalar curvatures $R$ and $R^{\gamma}$ for the noncommutative two torus, up to an overall factor of $-\pi/\tau_2$,  are given by \cite{CM1,FK2}
\begin{align*}
R&=R_1(\log \Delta)\Big (h^{ij}\delta_i\delta_j(\log k)\Big)+R_2(\log \Delta_1,\log \Delta_2)\Big(h^{ij}\delta_i(\log k)\delta_j(\log k)\Big)\\
&-iW(\log \Delta_1,\log \Delta_2)\Big(-i\epsilon^{ij}\delta_i(\log k)\delta_j(\log k)\Big)
\end{align*}
where
\begin{align*}
&R_1(x)=\frac{\frac{1}{2}-\frac{\sinh (x/2)}{x}}{\sinh ^2(x/4)},\\
&R_2(s,t)=-(1+\cosh ((s+t)/2))\times\\
&\frac{(s+t)(-t\cosh s+s\cosh t)-(s-t)(s+t+\sinh s+\sinh t -\sinh(s+t))}{st(s+t)\sinh (s/2)\sinh (t/2)\sinh ^2((s+t)/2)},\\
&W=\frac{(-s-t+t\cosh s+s\cosh t+\sinh s+\sinh t -\sinh(s+t)}{st(s+t)\sinh (s/2)\sinh (t/2)\sinh ((s+t)/2)}
\end{align*}
and
\begin{align*}
R^{\gamma}&=R^{\gamma}_1(\log \Delta)\Big (h^{ij}\delta_i\delta_j(\log k)\Big)+R^{\gamma}_2(\log \Delta_1,\log \Delta_2)\Big(h^{ij}\delta_i(\log k)\delta_j(\log k)\Big)\\
&-iW(\log \Delta_1,\log \Delta_2)\Big(-i\epsilon^{ij}\delta_i(\log k)\delta_j(\log k)\Big)
\end{align*}
where
\begin{align*}
&R^{\gamma}_1(x)=\frac{\frac{1}{2}+\frac{\sinh (x/2)}{x}}{\cosh ^2(x/4)},\\
&R^{\gamma}_2(s,t)=-(1-\cosh ((s+t)/2))\times\\
&\frac{(s+t)(-t\cosh s+s\cosh t)-(s-t)(s+t+\sinh s+\sinh t -\sinh(s+t))}{st(s+t)\sinh (s/2)\sinh (t/2)\sinh ^2((s+t)/2)},\\
&W=\frac{(-s-t+t\cosh s+s\cosh t+\sinh s+\sinh t -\sinh(s+t)}{st(s+t)\sinh (s/2)\sinh (t/2)\sinh ((s+t)/2)}.
\end{align*}
\end{theorem}
\begin{proof}
The terms of $b_2^1(\xi,-1)$ can be divided into two parts: those that do not contain $w$ which are exactly as those that appeared in \cite {FK2}, and new terms that contain $w$. One can check that these new terms with the parameter $w$ involved, up to an overall factor of $r/\tau_2$, are the following;

\begin{align*}
&
+2 \pi  w_1 b_0k\delta _1(k)b_0+2 i \pi  w_2 b_0k\delta _1(k)b_0+2 \pi  w_1 \tau _1 b_0k\delta _2(k)b_0+2 i \pi  w_2 \tau _1 b_0k\delta _2(k)b_0
\\&
+2 i \pi  w_1 \tau _2 b_0k\delta _2(k)b_0-2 \pi  w_2 \tau _2 b_0k\delta _2(k)b_0+2 \pi  w_1 b_0\delta _1(k)b_0k+2 i \pi  w_2 b_0\delta _1(k)b_0k
\\&
-2 \pi  r^2 w_1 b_0\delta _1(k)b_0^2k^3-2 i \pi  r^2 w_2 b_0\delta _1(k)b_0^2k^3+2 \pi  w_1 \tau _1 b_0\delta _2(k)b_0k+2 i \pi  w_2 \tau _1 b_0\delta _2(k)b_0k
\\&
+2 i \pi  w_1 \tau _2 b_0\delta _2(k)b_0k-2 \pi  w_2 \tau _2 b_0\delta _2(k)b_0k-2 \pi  r^2 w_1 \tau _1 b_0\delta _2(k)b_0^2k^3-2 i \pi  r^2 w_2 \tau _1 b_0\delta _2(k)b_0^2k^3
\\&
-2 i \pi  r^2 w_1 \tau _2 b_0\delta _2(k)b_0^2k^3+2 \pi  r^2 w_2 \tau _2 b_0\delta _2(k)b_0^2k^3-10 \pi  r^2 w_1 b_0^2k^3\delta _1(k)b_0-2 i \pi  r^2 w_2 b_0^2k^3\delta _1(k)b_0
\\
&
-10 \pi  r^2 w_1 \tau _1 b_0^2k^3\delta _2(k)b_0-2 i \pi  r^2 w_2 \tau _1 b_0^2k^3\delta _2(k)b_0-2 i \pi  r^2 w_1 \tau _2 b_0^2k^3\delta _2(k)b_0
\\&
+10 \pi  r^2 w_2 \tau _2 b_0^2k^3\delta _2(k)b_0+8 \pi  r^4 w_1 b_0^3k^5\delta _1(k)b_0+8 \pi  r^4 w_1 \tau _1 b_0^3k^5\delta _2(k)b_0-
8 \pi  r^4 w_2 \tau _2 b_0^3k^5\delta _2(k)b_0
\\&
-2 \pi  r^2 w_1 b_0k\delta _1(k)b_0^2k^2-2 i \pi  r^2 w_2 b_0k\delta _1(k)b_0^2k^2-2 \pi  r^2 w_1 \tau _1 b_0k\delta _2(k)b_0^2k^2-2 i \pi  r^2 w_2 \tau _1 b_0k\delta _2(k)b_0^2k^2
\\&
-2 i \pi  r^2 w_1 \tau _2 b_0k\delta _2(k)b_0^2k^2+2 \pi  r^2 w_2 \tau _2 b_0k\delta _2(k)b_0^2k^2-10 \pi  r^2 w_1 b_0^2k^2\delta _1(k)b_0k-2 i \pi  r^2 w_2 b_0^2k^2\delta _1(k)b_0k
\\&
+4 \pi  r^4 w_1 b_0^2k^2\delta _1(k)b_0^2k^3-10 \pi  r^2 w_1 \tau _1 b_0^2k^2\delta _2(k)b_0k-2 i \pi  r^2 w_2 \tau _1 b_0^2k^2\delta _2(k)b_0k-2 i \pi  r^2 w_1 \tau _2 b_0^2k^2\delta _2(k)b_0k
\\&
+10 \pi  r^2 w_2 \tau _2 b_0^2k^2\delta _2(k)b_0k+4 \pi  r^4 w_1 \tau _1 b_0^2k^2\delta _2(k)b_0^2k^3-4 \pi  r^4 w_2 \tau _2 b_0^2k^2\delta _2(k)b_0^2k^3
\\&
+4 \pi  r^4 w_1 b_0^2k^3\delta _1(k)b_0^2k^2
+4 \pi  r^4 w_1 \tau _1 b_0^2k^3\delta _2(k)b_0^2k^2-4 \pi  r^4 w_2 \tau _2 b_0^2k^3\delta _2(k)b_0^2k^2+8 \pi  r^4 w_1 b_0^3k^4\delta _1(k)b_0k+
\\&
+8 \pi  r^4 w_1 \tau _1 b_0^3k^4\delta _2(k)b_0k-8 \pi  r^4 w_2 \tau _2 b_0^3k^4\delta _2(k)b_0k-2 \pi |w|^2 b_0^2k^2+4 \pi  r^2 w_1^2 b_0^3k^4+4 \pi  r^2 w_2^2 b_0^3k^4.
\end{align*}
After integration by parts, up to overall factor of $\frac{\pi}{\tau_2} k^{-1}$, one can find that this is equal to
\begin{align*}
&\Big(w\mathcal D_0+w\mathcal D_0 \Delta^{1/2}-w\mathcal D_0\Delta^{1/2}+w\mathcal D_1\Delta^{1/2}-5w_1\mathcal D_1-iw_2\mathcal D_1+4 w_1\mathcal D_2-w\mathcal D_0
w\mathcal D_1-
\\&5w_1\mathcal D_1 \Delta^{1/2}-iw_2\mathcal D_1 \Delta^{1/2}+4w_1\mathcal D_1\Delta^{1/2}-4w_1\mathcal D_2\Delta^{1/2}+4w_1\mathcal D_1-4w_1\mathcal D_2+4w_1\mathcal D_2\Delta^{1/2}\Big) (\delta_1(k))+\\
&\Big(w\tau\mathcal D_0+w\tau\mathcal D_0 \Delta^{1/2}-w\tau\mathcal D_0\Delta^{1/2}+w\tau\mathcal D_1\Delta^{1/2}-5(w\tau)_1\mathcal D_1-i(w\tau)_2\mathcal D_1+4 (w\tau)_1\mathcal D_2-w\tau\mathcal D_0
w\mathcal D_1-
\\&5(w\tau)_1\mathcal D_1 \Delta^{1/2}-i(w\tau)_2\mathcal D_1 \Delta^{1/2}+4(w\tau)_1\mathcal D_1\Delta^{1/2}-4(w\tau)_1\mathcal D_2\Delta^{1/2}+4(w\tau)_1\mathcal D_1-4(w\tau)_1\mathcal D_2+\\
&4(w\tau)_1\mathcal D_2\Delta^{1/2}\Big) (\delta_2(k)),
\end{align*}
where
\[
(w\tau)_1=Re(w\tau),\qquad(w\tau)_2=Im (w\tau).
\]
One can easily find that the above expression vanishes. Indeed,  it will simplify to
\[\Big((w\mathcal D_1-w_1\mathcal D_1-iw_2\mathcal D_1)+(w\tau\mathcal D_1-(w\tau)_1\mathcal D_1-i(w\tau)_2\mathcal D_1)\Big)(1+\Delta^{1/2})=0.
\]

\end{proof}

\section{The nontrivial bundle $\SR$}
Let $\mathcal S(\mathbb R)$ denote the space of rapidly decreasing schwartz class functions on $\mathbb R$. It is a right $\At^{\infty}$-module \cite{C0,C1,C5}.
We fix a Powers-Rieffel projection $e$ such that $e\At^{\infty}\simeq \SR$ as right $\At^{\infty}$-modules. Then $e\delta_1$ and $e\delta_2$ will define a Grassmannian connection on $e\At^{\infty}$ and one has \cite{C0}
\[
2\pi i \tau_0(e[\delta_1(e),\delta_2(e)])=1.
\]
Note that the difference between this formula and the one in \cite{C0} is because of our convention on derivations $\delta_i$. One can adapt Connes' pseudodifferential calculus for this case. In fact, one has $End_{\Ats}(e\Ats)=e\Ats e$.  Suffices to say that symbols are smooth maps $
\rho :\mathbb R^2\to e\Ats e$ with appropriate growth condition. For example for any  pseudodifferential operator $P_{\rho}$ on $\At^{\infty}$, one can see that $eP_{\rho} e$ defines a pseudodifferential operator on $e\Ats$.

We introduce the twisted Dolbeault operator $\del_E:e\At^{\infty}\to e\At^{\infty}$ as
\[
\del_E=e(\delta_1+i\delta_2)=e\del .
\]
The hermitian structure of $e\At^{\infty}$ is given by 
\[
(\xi, \eta)=\eta^*\xi
\]
and the connection $\nabla_i=e\delta_i$, $i=1,2$ is compatible with this hermitian structure, i.e.
\[
(\nabla_i \xi, \eta)-(\xi,\nabla_i \eta)=\delta_i(\xi,\eta), \quad i=1,2.
\]
This implies that
\[
\del_E^*=e(\delta_1-i\delta_2)=e\del^*.
\]
 One can define the Laplacian on (0,0) sections as 
\[
\Delta^0_E=\del_E^*\del_E.
\]
Therfore
\[
\Delta^0_E=e\del^* e\del =e\delta_1^2+e\delta_2^2+ie[\delta_1(e),\delta_2(e)].
\]
The symbol of this operator is given by
\[
\sigma(\Delta^0_E)=a_0+a_1+a_2,
\]
where
\begin{align*}
&a_2=e\xi_1^2+e\xi_2^2,\qquad a_1=0,\qquad a_0=ie[\delta_1(e),\delta_2(e)].
\end{align*}
The Laplacian on  (1,0) sections is given by $\Delta^1_E=\del_E\del_E^*$,
and its symbol is
\[
\sigma(\Delta^1_E)=a_0+a_1+a_2,
\]
where
\begin{align*}
&a_2=e\xi_1^2+e\xi_2^2,\qquad a_1=0,\qquad a_0=-ie[\delta_1(e),\delta_2(e)].
\end{align*}
It is not difficult to find the $b_2$ terms of $\Delta^0$ and $\Delta^1$ in this case. For example, for $b_2(\Delta^1)$, after the polar change of coordinate $(\xi_1,\xi_2)\to (r,\theta)$ and integration over $\theta$ one  has: 
\begin{align*}
&2 \pi  r^2 e b_0^2\delta _1^2(e)b_0+2 \pi  r^2 e b_0^2 \delta _2^2(e) b_0-4 \pi  r^4 e b_0^3 \delta _1^2(e) b_0-4 \pi  r^4 e b_0^3 \delta _2^2(e) b_0+
\\
&2 i \pi  b_0 e \delta _1 (e) \delta _2 (e) b_0-2 i \pi  b_0 e \delta _2 (e) \delta _1 (e) b_0-8 \pi  r^4 e b_0^2 \delta _1 (e) b_0 \delta _1 (e) b_0-8 \pi  r^4 e b_0^2 \delta _2 (e) b_0 \delta _2 (e) b_0+
\\
&8 \pi  r^6 e b_0^3 \delta _1 (e) b_0 \delta _1 (e) b_0+8 \pi  r^6 e b_0^3 \delta _2 (e) b_0 \delta _2 (e) b_0+4 \pi  r^6 e b_0^2 \delta _1 (e) e b_0^2 \delta _1 (e) b_0+4 \pi  r^6 e b_0^2 \delta _2 (e) e b_0^2 \delta _2 (e) b_0
\end{align*}
With a similar computation one gets
\[
B_2(\Delta^0_{E})-B_2(\Delta^1_{E})=-2\pi i \tau_0(e[\delta _1(e),\delta _2(e)])
\]

The map $W: e\Ats\to( e\Ats)_{\varphi}$, given by $W(s)=sk$ is an  isometry. Here $( e\Ats)_{\varphi}$ is the completion of  $e\Ats$ with respect to the inner product $(\xi,\eta)_{\varphi}=\tau_0 (\eta ^*\xi  k^{-2})$. One can see that
\[
\del_{E,\varphi}\circ W=\del_E \circ R_k,
\]
where the operator $\del_{E,\varphi}$ is $\del_E$ but considered on $( e\Ats)_{\varphi}$. This yields that $W^* \Delta_{E,\varphi} W=R_k^*\Delta_E R_k$, that is $\Delta_{E,\varphi}$ is unitarily equivalent to $R_k^*\Delta_E R_k$. We employ a variational argument as in \cite{CM1}, Theorem 2.2 (cf. also \cite{CCT}). Letting
\[
\triangle^s_E=e^{sh/2}\triangle_E e^{sh/2},
\]
one has 
\[
\frac{d}{ds}\triangle^s_E=\frac{1}{2}(h \triangle^s_E+\triangle^s_E h).
\] 
We need to bear in mind that all multiplications are considered as right multiplication operators. Then
\[
\frac{d}{ds}\text{Tr}e^{-t\triangle^s_E}=-t\text{Tr}(h\triangle^s_E e^{-t\triangle^s_E})=t\frac{d}{dt}\text{Tr}(he^{-t\triangle^s_E}).
\]
Assuming that
\[
\text{Tr}(e^{-t\triangle^s_E})\sim t^{-1}\sum_{n=0}^{\infty}a_n(\triangle^s_E) t^{n/2},
\] 
and
\[
\text{Tr}(he^{-t\triangle^s_E})\sim t^{-1}\sum_{n=0}^{\infty}a_n(h,\triangle^s_E) t^{n/2},
\]
one has
\[
\frac{d}{ds}a_j(\triangle^s_E)=(\frac{j}{2}-1)a_j(h,\triangle^s_E),
\]
and hence
\[
\frac{d}{ds}a_2(\triangle^s_E)=0.
\]
This shows that the $B_2$ term of $\Delta_{E,\varphi}$ is independent of $\varphi$ and therefore 
\[
B_2(\Delta^0_{E,\varphi})-B_2(\Delta^1_{E,\varphi})=-2\pi i \tau_0(e[\delta _1(e),\delta _2(e)])=-1,
\]
which proves the following theorem.
\begin{theorem}
With the above notation, the index of the operator $\del_{E,\varphi}$ is independent of the conformal class of the metric and
\[
\text{Index}\,(\del_{E,\varphi})=-1.
\]
\end{theorem}

$^1$ \emph{E-mail}: masoud@uwo.ca\\
$^2$ \emph{E-mail}: amotadel@uwo.ca


\begin{thebibliography}{99}
\bibitem{CCT} A. Chamseddine, A. Connes, \emph{Scale invariance in the spectral action}. J.
Math. Phys. 47 (2006), N.6, 063504, 19 pp.
\bibitem{CohC} P. B. Cohen, A. Connes, \emph{Conformal geometry of the irrational rotation algebra}. Preprint MPI (92-93).
\bibitem{C0}A. Connes, \emph{$C^*$-alg\`ebres et g\'eom\'etrie diff\'erentielle}. C.R. Acad. Sc. Paris, t. 290, S ́rie A, 599-604 (1980).
\bibitem{C1} A. Connes, \emph{Noncommutative geometry}. Academic Press (1994).
\bibitem{C4} A. Connes, \emph{Noncommutative differential geometry}. Inst. Hautes Etudes Sci. Publ. Math. No. 62 (1985), 257-360.
\bibitem{C5} A. Connes, \emph{The action functional in noncommutative geometry}. Comm. Math. Phys. 117, no. 4, 673-683 (1988).
\bibitem{C7} A. Connes, M. Marcolli, \emph{Noncommutative Geometry, Quantum Fields and Motives}. American Mathematical Society Colloquium Publications, 55 (2008).
\bibitem{CM1} A. Connes, H. Moscovici, \emph{Modular curvature for noncommutative two-tori},  arXiv:1110.3500.
\bibitem{CM2}A. Connes, H. Moscovici, \emph{The local index formula in noncommutative geometry}. Geom. Funct. Anal. 5, no. 2, 174–243 (1995).
\bibitem{CM3} A. Connes, H. Moscovici, \emph{Type III and spectral triples. Traces in number theory, geometry and quantum fields}, 57-71, Aspects Math., E38, Friedr. Vieweg, Wiesbaden (2008).
\bibitem{CT} A. Connes, P. Tretkoff, \emph{The Gauss-Bonnet theorem for the noncommutative two torus},
arXiv:0910.0188. Noncommutative Geometry, Arithmetic,
and Related Topics, Johns Hopkins University press, Proceedings of the
Twenty-First Meeting of the Japan-U.S. Mathematics Institute edited by
Caterina Consani and Alain Connes (2011).
\bibitem{FK} F. Fathizadeh, M. Khalkhali, \emph{The Gauss-Bonnet theorem for noncommutative two Tori with a general conformal structure}, Journal of Noncommutative Geometry 6 (2012), no. 3, 457--480,  arXiv:1005.4947 .
\bibitem{FK2} F. Fathizadeh, M. Khalkhali,\emph{ Scalar curvature for the noncommutative two torus},  arXiv:1110.3511, to appear in Journal of Noncommutative Geometry.
\bibitem{G} P. Gilkey, \emph{Invariance theory, the heat equation, and the Atiyah-Singer index theorem. Mathematics Lecture Series}, 11. Publish or Perish, Inc., Wilmington, DE (1984).
\end{thebibliography}
\end{document}